\documentclass[a4paper,11pt]{article}

\usepackage{amsmath,amssymb}
\usepackage{amscd}
\usepackage{mathrsfs}
\usepackage{amsthm}
\usepackage{color}
\usepackage{url}

\theoremstyle{plain}

\newtheorem{theorem}{\bf Theorem}[section]
\newtheorem{lemma}[theorem]{\bf Lemma}
\newtheorem{proposition}[theorem]{\bf Proposition}
\newtheorem{corollary}[theorem]{\bf Corollary}

\theoremstyle{definition}
\newtheorem{definition}[theorem]{\bf Definition}

\newtheorem{remark}[theorem]{\bf Remark}

\newcommand{\eqa}[1]{
\begin{align*}
#1
\end{align*}}

\newcommand{\nai}[2]{\langle #1,#2\rangle}
\newcommand{\dom}[1]{{{\rm{dom}}{(#1)}}}

  \makeatletter
  \newcommand{\subsubsubsection}{\@startsection{paragraph}{4}{\z@}%
    {1.0\Cvs \@plus.5\Cdp \@minus.2\Cdp}%
    {.1\Cvs \@plus.3\Cdp}%
    {\reset@font\sffamily\normalsize}
  }
  \makeatother
\title{On Borel Equivalence Relations Related To Self-Adjoint Operators}

\author{Hiroshi Ando@\and Yasumichi Matsuzawa}

\begin{document}
\maketitle
\begin{abstract} 
In a recent work, we initiated the study of Borel equivalence relations defined on the Polish space ${\rm{SA}}(H)$ of self-adjoint operators on a Hilbert space $H$, focusing on the difference between bounded and unbounded operators. In this paper, we extend the analysis and show how the difficulty of specifying the domains of self-adjoint operators is reflected in Borel complexity of associated equivalence relations.
More precisely, we show that the equality of domains, regarded as an equivalence relation on ${\rm{SA}}(H)$ is continously bireducible with the orbit equivalence relation of the standard Borel group $\ell^{\infty}(\mathbb{N})$ on $\mathbb{R}^{\mathbb{N}}$. Moreover, we show that generic self-adjoint operators have purely singular continuous spectrum equal to $\mathbb{R}$.
\end{abstract}
\section{Introduction}
In the recent paper \cite{AM14}, the authors have studied Borel complexity of various equivalence relations defined on the space ${\rm{SA}}(H)$ of all (not necessarily bounded) self-adjoint operators on a separable and infinite-dimensional Hilbert space $H$ equipped with the strong resolvent topology (SRT). 
One major difference between bounded and unbounded operators is that due to the domain problems, ${\rm{SA}}(H)$ is not even a vector space: recall that the {\it sum} of self-adjoint operators $A,B$ is defined as the operator $C$ with $\dom{C}=\dom{A}\cap \dom{B}$ and $C\xi:=A\xi+B\xi,\ \xi\in \dom{C}$. 
In general, there is no reason to expect that $C$ is densely defined even if $\dom{A},\dom{B}$ are dense. 
In fact, Israel \cite{Israel04} has shown that if $A\in {\rm{SA}}(H)$ has empty essential spectrum, then the set of all unitaries $u$ satisfying $\dom{A}\cap u\cdot \dom{A}=\{0\}$ forms a norm dense $G_{\delta}$ subset of the unitary group $\mathcal{U}(H)$. Thus $\dom{A+uAu^*}=\{0\}$ for norm-generic $u$.   
 Therefore, it is natural to expect that the domain equivalence relation \[AE_{\rm{dom}}^{{\rm{SA}}(H)}B\Leftrightarrow \dom{A}=\dom{B}\] has a high degree of complexity. 
In this paper, we determine its exact Borel complexity by showing that  $E_{\rm{dom}}^{{\rm{SA}}(H)}$ is an $F_{\sigma}$ (but not $K_{\sigma}$) equivalence relation, and that it is continuously bireducible (see $\S2$ for the definition) with the $\ell^{\infty}(\mathbb{N},\mathbb{R})$-orbit equivalence relation $E_{\ell^{\infty}}^{\mathbb{R}^{\mathbb{N}}}$ defined on $\mathbb{R}^{\mathbb{N}}$ by \[(a_n)_{n=1}^{\infty}E_{\ell^{\infty}}^{\mathbb{R}^{\mathbb{N}}}(b_n)_{n=1}^{\infty}\Leftrightarrow \sup_{n\in \mathbb{N}}|a_n-b_n|<\infty,\ \ \ (a_n)_{n=1}^{\infty},(b_n)_{n=1}^{\infty}\in \mathbb{R}^{\mathbb{N}}.\]
Since Rosendal \cite[Proposition 19]{Rosendal05} has shown that $E_{\ell^{\infty}}^{\mathbb{R}^{\mathbb{N}}}$ is universal for $K_{\sigma}$-equivalence relations, $E_{\rm{dom}}^{{\rm{SA}}(H)}$ also enjoys this property. Moreover, since by this universality the notorious $K_{\sigma}$ equivalence relation $E_1$ (see $\S$3) is Borel reducible to $E_{\rm{dom}}^{{\rm{SA}}(H)}$, $E_{\rm{dom}}^{{\rm{SA}}(H)}$ is not Borel reducible to any orbit equivalence relation of a Borel action of a Polish group, by the Kechris-Louveau Theorem \cite[Theorem 4.2]{KechrisLouveau97}. Moreover, we show that the related equivalence relation $E_{{\rm{dom}},u}^{{\rm{SA}}(H)}$ (unitary equivalence of domains) given by 
\[AE_{{\rm{dom}},u}^{{\rm{SA}}(H)}B\Leftrightarrow \exists u\ \text{unitary}\ [u\cdot \dom{A}=\dom{B}]\]
is Borel reducible to a $K_{\sigma}$ equivalence relation, whence $E_{{\rm{dom}},u}^{{\rm{SA}}(H)}\le_B E_{\rm{dom}}^{{\rm{SA}}(H)}$ as a corollary. Finally, we strengthen our previous genericity result \cite[Theorem 3.17 (1)]{AM14} that elements in ${\rm{SA}}(H)$ which have essential spectrum $\mathbb{R}$, form a dense $G_{\delta}$ set. Namely we prove that elements in ${\rm{SA}}(H)$ which have purely singular continuous spectrum $\mathbb{R}$, forms a dense $G_{\delta}$ set in ${\rm{SA}}(H)$. This shows that although every self-adjoint operator can be approximated by diagonal operators (Weyl-von Neumann Theorem),  generic self-adjoint operators have rather pathological spectral properties (cf. \cite{ChokskiNadkarni,LatrachPaoliSimonnet}). The proof is based on Simon's Wonderland Theorem \cite{Simon95}.  
\section{Preliminaries}
We refer the reader to \cite[$\S 2$]{AM14} for relevant definitions and notation. Basic facts about operator theory (resp. descriptive set theory) can be found in \cite{Schmudgen} (resp. in \cite{Gao09,Hjorth00,Kechris96}). Below we give some definitions here for convenience. Let $H$ be a separable infinite-dimensional Hilbert space. 
\begin{definition}
The {\it strong resolvent topology} (SRT) on the space ${\rm{SA}}(H)$ of all self-adjoint operators on $H$ is the coarsest topology which makes the map ${\rm{SA}}(H)\ni A\mapsto (A-i)^{-1}\in \mathbb{B}(H)$ continuous with respect to the strong operator topology (SOT). 
\end{definition}
 ${\rm{SA}}(H)$ is Polish with respect to SRT. The domain of $A\in {\rm{SA}}(H)$ is written as $\dom{A}$. 
\begin{definition}
Let $E$ (resp. $F$) be equivalence relations on a Polish space $X$ (resp. $Y$). We say that $E$ is {\it Borel (resp. continuously) reducible} to $F$, denoted $E\le_BF$ (resp. $E\le_c F$), if there is a Borel (resp. continuous) map $f\colon X\to Y$ which is a reduction of $E$ to $F$ (i.e., $xEy\Leftrightarrow f(x)Ff(y)$ holds for $x,y\in X$). 
If moreover $f$ is injective, we say that $E$ is {\it Borel (resp. continuously) embeddable} into $F$, denoted $E\sqsubseteq_BF$ (resp. $E\sqsubseteq_cF$).  
We say that $E$ is {\it Borel (resp. continuously) bireducible} with $F$, if $E\le_BF$ and $F\le_BE$ (resp. $E\le_cF$ and $F\le_cE$) hold. In this case we write $E\sim_BF$ (resp. $E\sim_cF$). 
\end{definition}
In the next section we consider the following three equivalence relations.    
\begin{definition}We define $E_{\ell^{\infty}}^{\mathbb{R}^{\mathbb{N}}}$, $E_{\rm{dom}}^{{\rm{SA}}(H)}$ and $E_{{\rm{dom}},u}^{{\rm{SA}}(H)}$ by:
\begin{itemize}
\item[(1)] The equivalence relation $E_{\ell^{\infty}}^{\mathbb{R}^{\mathbb{N}}}$ on the Polish space $\mathbb{R}^{\mathbb{N}}$ is  the orbit equivalence relation of the action of the standard Borel group $\ell^{\infty}=\ell^{\infty}(\mathbb{N})$ on $\mathbb{R}^{\mathbb{N}}$ by addition. In other words, we have   $(a_n)_{n=1}^{\infty}E_{\ell^{\infty}}^{\mathbb{R}^{\mathbb{N}}}(b_n)_{n=1}^{\infty}\Leftrightarrow \sup_{n\in \mathbb{N}}|a_n-b_n|<\infty$ for $(a_n)_{n=1}^{\infty},(b_n)_{n=1}^{\infty}\in \mathbb{R}^{\mathbb{N}}$.
\item[(2)] The equivalence relation $E_{\rm{dom}}^{{\rm{SA}}(H)}$ on ${\rm{SA}}(H)$ is given by $AE_{\rm{dom}}^{{\rm{SA}}(H)}B\Leftrightarrow \dom{A}=\dom{B}$.
\item[(3)] The equivalence relation $E_{{\rm{dom}},u}^{{\rm{SA}}(H)}$ on ${\rm{SA}}(H)$ is given by 
$AE_{{\rm{dom}},u}^{{\rm{SA}}(H)}B\Leftrightarrow \exists u\in \mathcal{U}(H)\ [u\cdot \dom{A}=\dom{B}]$. 
\end{itemize}  
\end{definition}
We also recall a result on operator ranges. Recall that a subspace $\mathcal{R}\subset H$ is an {\it operator range} in $H$, if  $\mathcal{R}$ is equal to the range $\text{Ran}(T)$ for some $T\in \mathbb{B}(H)$. 
We may choose $T$ to be self-adjoint with $0\le T\le 1$. In this case, we set $H_n:=E_T((2^{-n-1},2^{-n}])H\ (n=0,1,\cdots)$. Then $H_n$ are mutually orthogonal closed subspaces of $H$ with $H=\bigoplus_{n=0}^{\infty}H_n$ (by the density of $\mathcal{R}$). 
$\{H_n\}_{n=0}^{\infty}$ are called the {\it associated subspaces for} $T$ (see \cite[$\S$3]{FillmoreWilliams}  for details). Since we are only concerned with dense operator ranges, we state the following result \cite[Theorem 3.3]{FillmoreWilliams} for dense operator ranges (in this case the condition (1) of  the cited theorem is automatic). 
\begin{theorem}[K\"othe, Fillmore-Williams]\label{thm: Koethe theorem}
Let $\mathcal{R}$ and $\mathcal{S}$ be dense operator ranges in $H$ with associated subspaces $\{H_n\}_{n=0}^{\infty}$ and $\{K_n\}_{n=0}^{\infty}$, respectively. Then there exists $u\in \mathcal{U}(H)$ such that $u\mathcal{R}=\mathcal{S}$, if and only if there exists $k\ge 0$ such that for each $n\ge 0$ and $l\ge 0$, one has
\eqa{
\dim(H_n\oplus  \cdots \oplus H_{n+l})& \le \dim (K_{n-k}\oplus \cdots \oplus K_{n+l+k})\\
\dim(K_n\oplus  \cdots \oplus K_{n+l})& \le \dim (H_{n-k}\oplus \cdots \oplus H_{n+l+k}),
}
where we use the convention $H_m=K_m=\{0\}$ for $m<0$.
\end{theorem}

Finally, for $A\in {\rm{SA}}(H)$, we denote by $\sigma_{\rm{p}}(A)$, $\sigma_{\rm{ac}}(A)$ and $\sigma_{\rm{sc}}(A)$ the set of eigenvalues, absolutely continuous spectrum, and singular continuous spectrum of $A$, respectively (see \cite[$\S$VII.2]{ReedSimonI}). We put $\sigma_{\rm{ac}}(A)=\emptyset$ (resp. $\sigma_{\rm{sc}}(A)=\emptyset$) if there is no absolutely continuous part (resp. singular continuous part) of $A$, and we say that $A$ has {\it purely singular continuous spectrum}, if $\sigma_{\rm{p}}(A)=\emptyset=\sigma_{\rm{ac}}(A)$ holds. 
\section{Main Results}
Now we state the main result.
\begin{theorem}\label{thm: E_dom is borel bireducible to ell^infty}
$E_{{\rm{dom}}}^{{\rm{SA}}(H)}$ is an $F_{\sigma}$ equivalence relation which is continuously bireducible with $E_{\ell^{\infty}}^{\mathbb{R}^{\mathbb{N}}}$. 
\end{theorem}
Before going to the proof, let us state an immediate corollary. We need two important results. Recall that a subspace of a topological space is called $K_{\sigma}$ or $\sigma$-compact, if it is a countable union of compact subsets.  
First, Rosendal \cite[Proposition 19]{Rosendal05} has shown that 
\begin{theorem}[Rosendal]\label{thm: Rosendal theorem}$E_{\ell^{\infty}}^{\mathbb{R}^{\mathbb{N}}}$ is universal for $K_{\sigma}$ equivalence relations in the sense that any $K_{\sigma}$ equivalence relation on a Polish space is Borel reducible to $E_{\ell^{\infty}}^{\mathbb{R}^{\mathbb{N}}}$.
\end{theorem}
Secondly, recall the $K_{\sigma}$ equivalence relation $E_1$ on $\mathcal{C}^{\mathbb{N}}$ (where $\mathcal{C}=2^{\mathbb{N}}$) defined by 
\[(a_n)_{n=1}^{\infty}E_1(b_n)_{n=1}^{\infty}\Leftrightarrow \exists N\in \mathbb{N}\ \forall n\ge N \ [a_n=b_n].\]
Since $\mathcal{C}$ and $\mathbb{R}$ are Borel isomorphic, $E_1$ may alternatively be defined (when talking about Borel reducibility) as the tail equivalence relation on $\mathbb{R}^{\mathbb{N}}$. 
 Kechris-Louveau \cite[Theorem 4.2]{KechrisLouveau97} has shown that $E_1$ is an obstruction for a given equivalence relation to be Borel reducible to orbit equivalence:
\begin{theorem}[Kechris-Louveau]\label{thm: KechrisLouveau}
$E_1\not\le_BE_G^X$ for any Polish group $G$ and Polish $G$-space $X$. 
\end{theorem}
Here, $E_G^X$ stands for the orbit equivalence relation associated with the Borel $G$-action. 
Since there are many orbit equivalence relations that are turbulent (in the sense of \cite{Hjorth00}) and Borel reducible to $E_{\ell^{\infty}}^{\mathbb{R}^{\mathbb{N}}}$ (e.g. $\ell^p(\mathbb{N})\ (1\le p<\infty)$ actions on $\mathbb{R}^{\mathbb{N}}$), Theorems \ref{thm: E_dom is borel bireducible to ell^infty},  \ref{thm: Rosendal theorem} and \ref{thm: KechrisLouveau} imply that:
\begin{corollary}\label{cor: E_dom is K_sigma universal}
$E_{{\rm{dom}}}^{{\rm{SA}}(H)}$ is universal for $K_{\sigma}$-equivalence relations. In particular, it is unclassifiable by countable structures, not Borel reducible to orbit equivalence relation of any Polish group action.
\end{corollary}
Now we prove Theorem \ref{thm: E_dom is borel bireducible to ell^infty} in few steps.
\begin{proposition}\label{prop: E_dom is F_sigma}$E_{\rm{dom}}^{{\rm{SA}}(H)}$ is an $F_{\sigma}$ equivalence relation which is not $K_{\sigma}$. 
\end{proposition} 
The proof relies on Douglas' range inclusion Theorem \cite{Douglas66} (cf. \cite[Theorem 2.1]{FillmoreWilliams}). 
\begin{theorem}[Douglas]\label{thm: Douglas theorem}
Let $A,B\in \mathbb{B}(H)$. Then ${\rm{Ran}}(A)\subset {\rm{Ran}}(B)$ holds if and only if 
there exists $\lambda>0$ such that $AA^*\le \lambda BB^*$.
\end{theorem}

\begin{proof}[Proof of Proposition \ref{prop: E_dom is F_sigma}]
It is clear that $\tau\colon {\rm{SA}}(H)^2\ni (A,B)\mapsto (B,A)\in {\rm{SA}}(H)^2$ is a homeomorphism. 
Define $\mathcal{S}:=\{(A,B)\in {\rm{SA}}(H)^2; \dom{A}\subset \dom{B}\}$. 
Since $E_{\rm{dom}}^{{\rm{SA}}(H)}=\mathcal{S}\cap \tau(\mathcal{S})$, it suffices to show that $\mathcal{S}$ is $F_{\sigma}$ in ${\rm{SA}}(H)^2$. 
For $A,B\in {\rm{SA}}(H)$, we have $\dom{A}=\text{Ran}((|A|+1)^{-1}), \dom{B}=\text{Ran}((|B|+1)^{-1})$. Therefore Theorem \ref{thm: Douglas theorem} shows that
\eqa{
\dom{A}\subset \dom{B}&\Leftrightarrow \exists \lambda>0\  [\ (|A|+1)^{-2}\le \lambda (|B|+1)^{-2}\ ]\\
&\Leftrightarrow \exists k\in \mathbb{N}\ \forall \xi\in H\ [\ \nai{\xi}{(|A|+1)^{-2}\xi}\le k \nai{\xi}{(|B|+1)^{-2}\xi} \ ].
}
Therefore $\mathcal{S}=\bigcup_{k\in \mathbb{N}}\bigcap_{\xi \in H}S_{k,\xi}$, where $S_{k,\xi}:=\{(A,B); \nai{\xi}{(|A|+1)^{-2}\xi}\le k \nai{\xi}{(|B|+1)^{-2}\xi}\}$.\\
It is easy to see that ${\rm{SA}}(H)\ni A\mapsto (|A|+1)^{-2}\in \mathbb{B}(H)$ is SRT-SOT continuous, hence each $S_{k,\xi}$ is SRT-closed. Therefore $\mathcal{S}$ is $F_{\sigma}$. The last assertion follows from the fact that ${\rm{SA}}(H)$ is not $K_{\sigma}$ (it contains a homeomorphic copy of $\mathbb{R}^{\mathbb{N}}$) and a well-known fact: note that if an equivalence relation $E$ on a Polish space $X$ is $K_{\sigma}$, then $X$ must be $K_{\sigma}$.
\end{proof}
\begin{proof}[Proof of Theorem \ref{thm: E_dom is borel bireducible to ell^infty}]
$E_{\rm{dom}}^{{\rm{SA}}(H)}$ is $F_{\sigma}$ but not $K_{\sigma}$ by Proposition \ref{prop: E_dom is F_sigma}. We show that $E_{\rm{dom}}^{{\rm{SA}}(H)}$ is continuously bireducible with $E_{\ell^{\infty}}^{\mathbb{R}^{\mathbb{N}}}$. We first show that $E_{\rm{dom}}^{{\rm{SA}}(H)}\le_c E_{\ell^{\infty}}^{\mathbb{R}^{\mathbb{N}}}$. Fix a dense countable subset $\{\xi_n\}_{n=1}^{\infty}$ of $H$. Given $A\in {\rm{SA}}(H)$, define $T_A:=(|A|+1)^{-2}$. 
Since $T_A$ is positive and $0$ is not an eigenvalue for $T_A$, $\nai{\xi_n}{T_A\xi_n}>0$ for every $n\in \mathbb{N}$. Moreover, $A\mapsto T_A$ is SRT-SOT continuous by functional calculus. Therefore we may define  a continuous map $\varphi\colon {\rm{SA}}(H)\to \mathbb{R}^{\mathbb{N}}$ by
\[\varphi(A):=(a_n(A))_{n=1}^{\infty},\ \ \ \ a_n(A):=\log (\nai{\xi_n}{T_A\xi_n}),\ \ \ \ A\in {\rm{SA}}(H),\ n\in \mathbb{N}.\]
We show that $\varphi$ is a reduction map. Let $A,B\in {\rm{SA}}(H)$. By the proof of Proposition \ref{prop: E_dom is F_sigma}, we have 
\eqa{
\dom{A}=\dom{B}&\Leftrightarrow \exists C_1>0\ \exists C_2>0\ [\ C_1T_B\le T_A\le C_2T_B\ ]\\
&\Leftrightarrow \exists C_1>0\ \exists C_2>0\ \forall n\in \mathbb{N}\ [\ C_1\nai{\xi_n}{T_B\xi_n}\le \nai{\xi_n}{T_A\xi_n}\le C_2\nai{\xi_n}{T_B\xi_n}\ ]\\
&\Leftrightarrow \exists C_1>0\ \exists C_2>0\ \forall n\in \mathbb{N}\ [\ \log C_1\le a_n(A)-a_n(B)\le \log C_2\ ]\\
&\Leftrightarrow \sup_{n\in \mathbb{N}}|a_n(A)-a_n(B)|<\infty\\
&\Leftrightarrow \varphi(A)E_{\ell^{\infty}}^{\mathbb{R}^{\mathbb{N}}}\varphi(B),
 }
which shows that $E_{\rm{dom}}^{{\rm{SA}}(H)}\le_c E_{\ell^{\infty}}^{\mathbb{R}^{\mathbb{N}}}$.

 Next we show that $ E_{\ell^{\infty}}^{\mathbb{R}^{\mathbb{N}}}\le_c E_{\rm{dom}}^{{\rm{SA}}(H)}$.  The proof is similar to the first part. Fix a complete orthonormal system (CONS) $\{\eta_n\}_{n=1}^{\infty}$ for $H$. 
For each $(x_n)_{n=1}^{\infty}\in \mathbb{R}^{\mathbb{N}}$, define $(\tilde{x}_n)_{n=1}^{\infty}\in \mathbb{R}_{\ge 0}^{\mathbb{N}}$ by
\[(\tilde{x}_{2n-1},\tilde{x}_{2n})=\begin{cases}(|x_n|,0) & (x_n\ge 0)\\
(0,|x_n|) & (x_n<0)\end{cases},\ \ \ \ \ n\in \mathbb{N}.\]
Thus $(1,-\frac{1}{2},4,0,\cdots)$ is mapped to $(1,0,0,\frac{1}{2},4,0,0,0,\cdots)$, etc. 
It is easy to see that $\mathbb{R}^{\mathbb{N}}\ni (x_n)_{n=1}^{\infty}\mapsto (\tilde{x}_n)\in \mathbb{R}_{\ge 0}^{\mathbb{N}}$ is an injective continuous map satisfying 
\begin{equation}
\sup_{n\in \mathbb{N}}|x_n-y_n|<\infty\Leftrightarrow \sup_{n\in \mathbb{N}}|\tilde{x}_n-\tilde{y}_n|<\infty,\ \ \ \ \ \ (x_n)_{n=1}^{\infty},(y_n)_{n=1}^{\infty}\in \mathbb{R}^{\mathbb{N}}.\label{eq: tilda map remembers the sign}
\end{equation}
We define $\psi\colon \mathbb{R}^{\mathbb{N}}\to {\rm{SA}}(H)$ by 
\[\psi(\alpha):=\sum_{n=1}^{\infty}\{\exp (\tfrac{1}{2}\tilde{x}_n)-1\}\nai{\eta_n}{\ \cdot\ }\eta_n,\ \ \ \ \alpha=(x_n)_{n=1}^{\infty}\in \mathbb{R}.\]
It is easy to see that $\psi$ is continuous, and 
\[T_{\psi(\alpha)}=(\psi(\alpha)+1)^{-2}=\sum_{n=1}^{\infty}\exp(-\tilde{x}_n)\nai{\eta_n}{\ \cdot\ }\eta_n, \ \ \ \ \alpha=(x_n)_{n=1}^{\infty}\in \mathbb{R}^{\mathbb{N}}.\] 
We show that $\psi$ is a reduction map. Given $\alpha=(x_n)_{n=1}^{\infty},\beta=(y_n)_{n=1}^{\infty}\in \mathbb{R}^{\mathbb{N}}$, we have (by (\ref{eq: tilda map remembers the sign}))
\eqa{
\dom{\psi(\alpha)}=\dom{\psi(\beta)}&\Leftrightarrow \exists C_1>0\ \exists C_2>0\ [\ C_1T_{\psi(\beta)}\le T_{\psi(\alpha)}\le C_2T_{\psi(\beta)}\ ]\\
&\Leftrightarrow \exists C_1>0\ \exists C_2>0\ \forall n\in \mathbb{N}\\
&\hspace{1.5cm}\ [\ C_1\exp(-\tilde{y}_n)\le \exp(-\tilde{x}_n)\le C_2\exp(-\tilde{y}_n)\ ]\\
&\Leftrightarrow \sup_{n\in \mathbb{N}}|\tilde{y}_n-\tilde{x}_n|<\infty\\
&\Leftrightarrow \alpha E_{\ell^{\infty}}^{\mathbb{R}^{\mathbb{N}}} \beta,
}
whence $E_{\ell^{\infty}}^{\mathbb{R}^{\mathbb{N}}}\le_c E_{\rm{dom}}^{{\rm{SA}}(H)}$. 
 This shows that $E_{\ell^{\infty}}^{\mathbb{R}^{\mathbb{N}}}$ is continuously bireducible with $E_{\rm{dom}}^{{\rm{SA}}(H)}$. 
\end{proof}
As another corollary to Theorem \ref{thm: E_dom is borel bireducible to ell^infty}, we prove that $E_{{\rm{dom}},u}^{{\rm{SA}}(H)}\le_B E_{\rm{dom}}^{{\rm{SA}}(H)}$. This is done by showing that $E_{{\rm{dom}},u}^{{\rm{SA}}(H)}$ is Borel reducible to a $K_{\sigma}$ equivalence relation. 
Regard $\mathbb{N}^*:=\mathbb{N}\cup \{\infty\}$ as a one-point compactification of $\mathbb{N}=\{1,2,\cdots\}$. 
Thus $\mathbb{N}^*$ is homeomorphic to $\{\frac{1}{n};n\in \mathbb{N}\}\cup \{0\}$ by $n\mapsto \frac{1}{n}\ (n\in \mathbb{N})$ and $\infty\mapsto 0$.  
Consider the compact Polish space $X:=\prod_{n=0}^{\infty}(\mathbb{N}^*\cup \{0\})$, and define  
$X_0:=\left \{(a_n)_{n=0}^{\infty}\in X;\ \sum_{n=0}^{\infty}a_n=\infty\right \}$. 
Then $X_0$ is a (dense) $G_{\delta}$ subspace of $X$, whence Polish. 
\begin{definition}
Define an equivalence relation $E_{\Sigma}$ on $X$ by 
$(a_n)_{n=0}^{\infty}E_{\Sigma}(b_n)_{n=0}^{\infty}$ if and only if there exists $k\ge 0$ such that for each $l\ge 0$ and $n\ge 0$, 
\[\sum_{i=0}^la_{n+i}\le \sum_{j=-k}^{l+k}b_{n+j}\text{\ \ and\ \ }\sum_{i=0}^lb_{n+i}\le \sum_{j=-k}^{l+k}a_{n+j}.\]
Here, we regard $a_n=b_n=0\ (n<0)$ and $\infty+n=n+\infty=\infty+\infty=\infty\ (n\in \mathbb{N})$. 
\end{definition}
\begin{proposition}\label{prop: E_dom,u is essentially K_sigma}
$E_{\Sigma}$ is a $K_{\sigma}$ equivalence relation, and $E_{{\rm{dom}},u}^{{\rm{SA}}(H)}\sim_B E_{\Sigma}|_{X_0}\ (\le_B E_{\Sigma})$. In particular, $E_{{\rm{dom}},u}^{{\rm{SA}}(H)}$ is Borel reducible to a $K_{\sigma}$ equivalence relation.  
\end{proposition}
We omit the proof of the next easy lemma.  
\begin{lemma}\label{lem: sum is continuous}
For $n,m\in \mathbb{N}\cup \{0\} (n\le m)$, the map $X\ni (a_k)_{k=0}^{\infty}\mapsto \sum_{k=n}^ma_k\in \mathbb{N}^*$ is continuous.
\end{lemma}

\begin{lemma}\label{lem: rank function is Borel}
Let $a,b\in \mathbb{R}, a<b$, and let $I=(a,b), [a,b)$ or $(a,b]$. Then the map ${\rm{SA}}(H)\ni A\mapsto {\rm{rank}}(E_A(I))\in \mathbb{N}^*$ is Borel. 
\end{lemma}
\begin{proof}
We show the case of $I=[a,b)$. Let $S_n:=\{A\in {\rm{SA}}(H);\ \text{rank}(E_A([a,b)))\le n\}\ (n\in \mathbb{N}\cup \{0\})$, $S_{\infty}:=\{A\in {\rm{SA}}(H);\ \text{rank}(E_A([a,b)))=\infty\}$. Then 
by a similar argument to the proof of  \cite[Proposition 3.18]{AM14} (especially that $S_{n,k}$ defined there is SRT-closed), it can be shown that $S_n$ is SRT-closed. 
Therefore $\{A\in {\rm{SA}}(H); \text{rank}(E_A([a,b)))=n\}=S_n\setminus S_{n-1}\ (n\ge 1)$ and $S_0$ are Borel. Then $S_{\infty}={\rm{SA}}(H)\setminus \bigcup_{n\ge 0}S_n$ is Borel too. Thus the map $A\mapsto {\rm{rank}}(E_A(I))$ is Borel.
\end{proof}

\begin{proof}[Proof of Proposition \ref{prop: E_dom,u is essentially K_sigma}]
It is easy to see that $\dom{A}=\dom{|A|+1}$ for every $A\in {\rm{SA}}(H)$, and $\dom{A}={\rm{Ran}}((|A|+1)^{-1})$. The associated subspaces for $T_A=(|A|+1)^{-1}$ are 
\[H_n(T_A)=E_{T_A}((2^{-n-1},2^n])H,\ \ \ \ n\ge 0.\]
Note that for $\lambda \in \sigma(A)$, 
\[(|\lambda|+1)^{-1}\in (2^{-n-1},2^n]\Leftrightarrow \lambda \in \underbrace{(1-2^{n+1},1-2^n]\cup [2^n-1,2^{n+1}-1)}_{=:I_n\cup J_n}.\]
Let $d_0(A):=\text{rank}(E_A(-1,1))$ and $d_n(A):=\dim H_n(T_A)=\text{rank}(E_A(I_n))+\text{rank}(E_A(J_n))\ (n\ge 1)$.\ 
By Lemma \ref{lem: rank function is Borel}, $d_n\colon {\rm{SA}}(H)\to \mathbb{N}^*$ is Borel for each $n\ge 0$.

Now, note that $E_{\Sigma}=\bigcup_{k=0}^{\infty}E_k$, where
\[E_k:=\bigcap_{l,n=0}^{\infty}\left \{((a_n)_{n=0}^{\infty},(b_n)_{n=0}^{\infty});\ \sum_{i=0}^la_{n+i}\le \sum_{j=-k}^{l+k}b_{n+j}\text{\ \ and\ \ }\sum_{i=0}^lb_{n+i}\le \sum_{j=-k}^{l+k}a_{n+j}\right \}.\]
It is immediate to see that $E_{\Sigma}$ is $K_{\sigma}$ because each $E_k$ is a closed subset of the compact space $X\times X$ by Lemma \ref{lem: sum is continuous}. 
Define a Borel map $\varphi\colon {\rm{SA}}(H)\to X_0$ by $\varphi(A):=(d_n(A))_{n=0}^{\infty}$.
Since $H$ is infinite-dimensional, $\varphi(A)\in X_0$. 
Moreover, $AE_{{\rm{dom}},u}^{{\rm{SA}}(H)}B$ if and only if $\varphi(A)E_{\Sigma}\varphi(B)$ by Theorem \ref{thm: Koethe theorem}. Therefore $E_{{\rm{dom}},u}^{{\rm{SA}}(H)}\le_B E_{\Sigma}|_{X_0}\le_BE_{\Sigma}$.
To show $E_{\Sigma}|_{X_0}\le_B E_{{\rm{dom}},u}^{{\rm{SA}}(H)}$, let 
\[X_{0,k}:=\left \{(a_n)_{n=0}^{\infty}\in X_0; \sharp \{n\in \mathbb{N}\cup \{0\};a_n=\infty\}=k\right \},\ \ \ \ \ k\in \mathbb{N}^*\cup \{0\}.\]
Note that each $X_{0,k}$ is a Borel subset of $X_0$: it is enough to see that $\widetilde{X}_{0,k}:=\bigcup_{i=0}^{k}X_{0,i}$ is closed in $X$. But if $\alpha_i=(a_{n,i})_{n=0}^{\infty}\in \widetilde{X}_{0,k}$ tends to $\alpha=(a_n)_{n=0}^{\infty}\in X_0$, then if $a_{n_1}=\cdots=a_{n_p}=\infty\ (n_1<n_2<\cdots<n_p)$, then by assumption there exists $i_0$ such that for each $i\ge i_0$ $a_{i,n_1}=\cdots=a_{i,n_p}=\infty$, so $p\le k$. Therefore $\alpha\in \widetilde{X}_{0,k}$, and $\widetilde{X}_{0,k}$ is closed.

Now define for each $k\in \mathbb{N}^*\cup \{0\}$ a Borel map $\psi_k\colon X_{0,k}\to {\rm{SA}}(H)$ by the following:\\ 
\textbf{Case $k=0$}.\\
Fix a CONS $\{\xi_n\}_{n=1}^{\infty}$ for $H$. For $\alpha=(a_n)_{n=0}^{\infty}\in X_{0,0}$, define 
\[\psi_0(\alpha):=\sum_{n=0}^{\infty}(2^{\frac{n}{2}}-1)e_n(\alpha),\]
where the projection $e_{n,0}(\alpha)$ is inductively defined as follows: $e_{0,0}(\alpha)$ is the projection onto $\text{span}\{\xi_1,\cdots,\xi_{a_0}\}$ (if $a_0\ge 1$) and $e_{0,0}(\alpha)=0$ otherwise, and for $k\ge 0$,
\[e_{k+1,0}(\alpha):=\text{projection onto span}\{\xi_{a_0+\cdots +a_k+1},\cdots \xi_{a_0+\cdots+a_k+a_{k+1}}\}\text{\ \ if\ \ }a_{k+1}\ge 1,\]
and $e_{k+1,0}(\alpha):=0$ otherwise. Then it is easy to see that $\psi_0\colon X_{0,0}\to {\rm{SA}}(H)$ is continuous, and $T_{\psi_0(\alpha)}=\sum_{n=0}^{\infty}2^{-n}e_{n,0}(\alpha)$. In particular, the rank of the associated subspace for $T_{\psi_0(\alpha)}$ is $d_n(\psi_0(\alpha))=a_n\ (n\ge 0)$.\\ \\
\textbf{Case $1\le k\le \infty$}.\\
Let $\alpha=(a_n)_{n=0}^{\infty}\in X_{0,k}$, and suppose that $a_{n_1}=\cdots=a_{n_k}=\infty\ (n_1<\cdots<n_k)$ (for $k=\infty$ case this means that $n_1<n_2<\cdots$ is an infinite sequence) and $a_n<\infty\ (n\notin \{n_1,\cdots,n_k\})$. 
Fix another CONS $\{\eta_n,\zeta_{p,n};n\ge 1,1\le p\le k\}$ for $H$, and define $\psi_k(\alpha)\in {\rm{SA}}(H)$ by
\[\psi_k(\alpha):=\sum_{n=0}^{\infty}(2^{\frac{n}{2}}-1)e_{n,k}(\alpha),\]
where the projection $e_{n,k}(\alpha)$ is defined as follows: define $(b_n)_{n=0}^{\infty}\in X_0$ inductively by 
\[b_0:=\begin{cases}\ a_0 & (a_0<\infty)\\
\ 0 & (a_0=\infty)\end{cases},\ \ \  b_{k+1}:=\begin{cases}b_k+a_{k+1} & (a_{k+1}<\infty)\\
\ \ \ \ b_k & (a_{k+1}=\infty)\end{cases},\ \ \ \ k\ge 0,\]
and then put $e_{0,k}(\alpha)=\text{projection onto span}\{\eta_{1},\cdots,\eta_{b_0}\}$ if $a_0<\infty$, and
 $e_{0,k}(\alpha):=\text{projection onto}\ \overline{\text{span}}\{\zeta_{1,i}\}_{i=1}^{\infty}$ if $a_0=\infty$. For $n\ge 1$, put 
\[
e_{n,k}(\alpha):=\begin{cases}\ \ \ \ \ \ \ \ \ 0 & (a_n=0)\\
\text{projection onto span}\{\eta_{b_{n-1}+1},\cdots,\eta_{b_n}\} & (0<a_{n}<\infty)\\
\text{projection onto}\ \overline{\text{span}}\{\zeta_{p,i}\}_{i=1}^{\infty} & (n=n_p)
\end{cases}.
\]
Again $\psi_k\colon X_{0,k}\to {\rm{SA}}(H)$ is continuous, and $d_n(\psi_k(\alpha))=a_n\ (n\ge 0)$. 

 Finally define $\psi\colon X_0\to {\rm{SA}}(H)$ by $\psi|_{X_{0,k}}:=\psi_k$. Then since each $X_{0,k}$ is Borel and $\psi_k$ is continuous on $X_{0,k}$, $\psi$ is Borel. Moreover, since $d_n(\psi(\alpha))=a_n (n\ge 0)$ for every $\alpha=(a_n)_{n=0}^{\infty}\in X_0$, it follows that $\alpha E_{\Sigma}\beta\Leftrightarrow \psi(\alpha)E_{{\rm{dom}},u}^{{\rm{SA}}(H)}\psi(\beta)$ for $\alpha,\beta\in X_0$. 
This shows that $E_{\Sigma}|_{X_0}\le_B E_{{\rm{dom}},u}^{{\rm{SA}}(H)}$. Therefore $E_{\Sigma}|_{X_0}\sim_B E_{{\rm{dom}},u}^{{\rm{SA}}(H)}$ holds.     
\end{proof}
\begin{corollary}\label{cor: E_dom,u is reducible to E_dom}
$E_{{\rm{dom}},u}^{{\rm{SA}}(H)}\le_B E_{\rm{dom}}^{{\rm{SA}}(H)}$ holds.
\end{corollary}
\begin{proof}By Proposition \ref{prop: E_dom,u is essentially K_sigma}, Theorems \ref{thm: E_dom is borel bireducible to ell^infty} and \ref{thm: Rosendal theorem}, it holds that $E_{{\rm{dom}},u}^{{\rm{SA}}(H)}\le_BE_{\ell^{\infty}}^{\mathbb{R}^{\mathbb{N}}}\sim_cE_{\rm{dom}}^{{\rm{SA}}(H)}$.
\end{proof}
\begin{remark}It is not clear whether $E_{{\rm{dom}}}^{{\rm{SA}}(H)}\le_B E_{{\rm{dom}},u}^{{\rm{SA}}(H)}$ holds. 
\end{remark}
\section{Generic $A$ has purely singular continuous spectrum $\mathbb{R}$}
In \cite[Theorem 3.17 (1)]{AM14}, we have shown a genericity result that the set $\{A\in {\rm{SA}}(H);\sigma_{\rm{ess}}(A)=\mathbb{R}\}$ is dense $G_{\delta}$ in ${\rm{SA}}(H)$. In this last section, we show that generic self-adjoint operators in fact have much more pathological spectral property:  
\begin{theorem}\label{thm: Wonderland for SA(H)}
The set $\mathcal{G}:=\{A\in {\rm{SA}}(H); \sigma_{\rm{p}}(A)=\sigma_{\rm{ac}}(A)=\emptyset,\ \sigma_{\rm{sc}}(A)=\mathbb{R}\}$ is dense $G_{\delta}$ in ${\rm{SA}}(H)$. 
\end{theorem}
The proof relies on the surprising theorem of Simon (which he calls ``Wonderland Theorem"). 
\begin{definition}\label{def: regular metric space of operators}\cite{Simon95} Let $(X,d)$ be a metric space of self-adjoint operators on $H$. $X$ is called a {\it regular metric space}, if $d$ is complete and generates a topology stronger than or equal to SRT. 
\end{definition}

\begin{theorem}[Simon's Wonderland Theorem]\label{thm: Simon Wonderland}
Let $(X,d)$ be a regular metric space of self-adjoint operators on $H$. Suppose that for some open interval $(a,b)$, 
\begin{itemize}
\item[{\rm{(1)}}] $\{A\in X;A\text{\ has purely continuous spectrum on\ }(a,b)\}$ is dense in $X$.
\item[{\rm{(2)}}] $\{A\in X;A\text{\ has purely singular spectrum on\ }(a,b)\}$ is dense in $X$.
\item[{\rm{(3)}}] $\{A\in X;A\text{\ has\ }(a,b)\text{\ in its spectrum}\}$ is dense in $X$.
\end{itemize} 
Then $\{A\in X; (a,b)\subset \sigma_{\rm{sc}}(A),\ (a,b)\cap \sigma_{\rm{p}}(A)=\emptyset,\ (a,b)\cap \sigma_{\rm{ac}}(A)=\emptyset\}$ is dense $G_{\delta}$ in $X$.
\end{theorem}

First we prove the density. 
\begin{proposition}\label{prop: density part: no eigenvalues}
The set $\{A\in {\rm{SA}}(H);\sigma_{\rm{p}}(A)=\sigma_{\rm{ac}}(A)=\emptyset\}$ is dense in ${\rm{SA}}(H)$.
\end{proposition}

\begin{lemma}\label{lem: singular continuous approximation of 1}
Let $H$ be an infinite-dimensional separable Hilbert space. There exists a sequence $\{A_n\}_{n=1}^{\infty}\subset {\rm{SA}}(H)$ with purely singular continuous spectrum, such that $A_n\stackrel{\rm{SRT}}{\to}1_H$.
\end{lemma}
\begin{proof}
Let $\mu$ be a singular continuous probability measure on $\mathbb{R}$. We identify $H=L^2(\mathbb{R},\mu)$, and define $A_n$ to be the multiplication by $f_n$, where $f_n(x):=\frac{1}{n}x+1\ \ (x\in \mathbb{R},n\in \mathbb{N})$. Then each $A_n$ has purely singular continuous spectrum, and $A_n\stackrel{\rm{SRT}}{\to}1_H$ by Lebesgue Dominated Convergence Theorem.
\end{proof}
\begin{proof}[Proof of Proposition \ref{prop: density part: no eigenvalues}]
Let $A\in {\rm{SA}}(H)$ and let $\mathcal{V}$ be an SRT-open neighborhood of $A$. By Weyl-von Neumann Theorem, there exists $A_0\in \mathcal{V}$ of the form $A_0=\sum_{n=1}^{\infty}a_n\nai{\xi_n}{\ \cdot\ }\xi_n$, where $\{a_n\}_{n=1}^{\infty}\subset \mathbb{R}$ and $\{\xi_n\}_{n=1}^{\infty}$ is an orthonormal basis for $H$. Let $e_n$ be the orthogonal projection of $H$ onto $\mathbb{C}\xi_n\ (n\in \mathbb{N})$. Let $k\in \mathbb{N}$. 
Choose a sequence of disjoint subsets $I_1^{(k)},I_2^{(k)},\cdots, I_k^{(k)}$ of $\mathbb{N}\setminus \{1,2,\cdots,k\}$ such that $|I_1^{(k)}|=|I_2^{(k)}|=\cdots=|I_k^{(k)}|=\infty$ and $\mathbb{N}\setminus \{1,\cdots,k\}=\bigsqcup_{i=1}^kI_i^{(k)}$. Then for each $1\le i\le k$, let $e_i^{(k)}$ be the projection of $H$ onto the closed linear span of $\{\xi_m;m\in I_i^{(k)}\}$, which is of infinite-rank.  Define a new operator $A_k\in {\rm{SA}}(H)$ by $A_k:=\sum_{n=1}^ka_ne_n+\sum_{n=1}^ka_ne_n^{(k)}$. Then $A_k\stackrel{k\to \infty}{\to}A_0$ (SRT), so that there exists $k_0\in \mathbb{N}$ such that $A_{k_0}\in \mathcal{V}$ holds. Now let $H_i\ (1\le i\le k_0)$ be the range of $e_i+e_i^{(k_0)}$, which is infinite-dimensional. Thus by Lemma \ref{lem: singular continuous approximation of 1}, we may find a sequence $\{A_{i,m}\}_{m=1}^{\infty}\subset {\rm{SA}}(H_i)$ with $\sigma_{\rm{p}}(A_{i,m})=\sigma_{\rm{ac}}(A_{m,i})=\emptyset\ (m\in \mathbb{N})$ such that $A_{i,m}\stackrel{m\to \infty}{\to}a_i1_{H_i}$ (SRT) for each $1\le i\le k_0$. Let $A_m:=\bigoplus_{i=1}^{k_0}A_{i,m}\in {\rm{SA}}(H)\ (m\in \mathbb{N})$. It follows that $A_m\stackrel{m\to \infty}{\to}A_{k_0}=\sum_{i=1}^{k_0}a_i(e_i+e_i^{(k_0)})\in \mathcal{V}$ (SRT), so that there exists $m_0\in \mathbb{N}$ such that $A_{m_0}\in \mathcal{V}$. Since $\sigma_{\rm{p}}(A_{m_0})=\sigma_{\rm{ac}}(A_{m_0})=\emptyset$ and $\mathcal{V}$ is arbitrary, the claim follows.  
\end{proof}

\begin{proof}[Proof of Theorem \ref{thm: Wonderland for SA(H)}]
For each $n\in \mathbb{N}$ define 
\[G_n:=\{A\in {\rm{SA}}(H); \sigma_{\rm{p}}(A)\cap (-n,n)=\sigma_{\rm{ac}}(A)\cap (-n,n)=\emptyset,\ (-n,n)\subset \sigma_{\rm{sc}}(A)\}.\] 
Since $\mathcal{G}=\bigcap_{n\in \mathbb{N}}G_n$, it suffices to show that each $G_n$ is dense $G_{\delta}$ in ${\rm{SA}}(H)$. We see that assumptions of Theorem \ref{thm: Simon Wonderland} are satisfied for $X={\rm{SA}}(H)$ with $(a,b)=(-n,n)$:\\
(1) and (2): the sets 
\[\{A\in {\rm{SA}}(H); A\text{\ has purely continuous spectrum on\ }(-n,n)\}\] and 
\[\{A\in {\rm{SA}}(H); A\text{\ has purely singular spectrum on\ }(-n,n)\}\]
are dense in ${\rm{SA}}(H)$, by Proposition \ref{prop: density part: no eigenvalues}.\\
(3): By \cite[Theorem 3.17 (1)]{AM14}, the set ${\rm{SA}}_{\rm{full}}(H)=\{A\in {\rm{SA}}(H);\sigma_{\rm{ess}}(A)=\mathbb{R}\}$ is a dense $G_{\delta}$ subset of ${\rm{SA}}(H)$. In particular, $\{A\in {\rm{SA}}(H); (-n,n)\subset \sigma(A)\}$ is dense in ${\rm{SA}}(H)$.\\
Therefore By Theorem \ref{thm: Simon Wonderland}, $G_n$ is dense $G_{\delta}$ in ${\rm{SA}}(H)$ for every $n\in \mathbb{N}$, which finishes the proof.
\end{proof}
\section*{Acknowledgments}
The authors would like to thank the anonymous referee for numerous suggestions which improved the presentation of the paper. 
HA was supported by the Danish National Research Foundation through the Centre for Symmetry and Deformation (DNRF92). YM was supported by KAKENHI 26800055 and 26350231. 

Hiroshi Ando\\
Department of Mathematical Sciences,\\
University of Copenhagen\\
Universitetsparken 5\\
2100 Copenhagen \O\ Denmark\\
ando@math.ku.dk\\ 
\url{http://andonuts.miraiserver.com/index.html}\\ \\
Yasumichi Matsuzawa\\
Department of Mathematics, Faculty of Education, Shinshu University\\
6-Ro, Nishi-nagano, Nagano, 380-8544, Japan\\
myasu@shinshu-u.ac.jp\\
\url{https://sites.google.com/site/yasumichimatsuzawa/home}
\end{document}